\documentclass [a4paper, 12pt, reqno]{amsart}
\usepackage [latin1]{inputenc}
\usepackage {a4}
\usepackage{amscd}
\usepackage{epsfig}
\usepackage{amssymb}
\usepackage{amsmath}
\usepackage{amsthm}
\usepackage[T1]{fontenc}
\usepackage{ae,aecompl}
\usepackage[arrow, matrix, curve]{xy}

\usepackage{geometry}
\newcommand{\C} {\ensuremath{\mathbb{C}}}

\newcommand{\OO}{\mathcal{O}}

\newcommand{\dq}{\overline{\partial}}

\DeclareMathOperator{\Reg}{Reg}
\DeclareMathOperator{\Sing}{Sing}

\DeclareMathOperator{\Dom}{Dom}

\newtheorem {satz} {Satz} [section]
\newtheorem {lem} [satz] {Lemma}
\newtheorem {cor} [satz] {Corollary}
\newtheorem {defn} [satz] {Definition}
\newtheorem {prop} [satz] {Proposition}

\newtheorem {thm} [satz] {Theorem}

\theoremstyle{remark}

\newtheorem {example} [satz] {Example}
\newtheorem {rem} [satz] {Remark}

\numberwithin{equation}{section}

\renewcommand{\theta}{\vartheta}

\title[Adjunction and extension of $L^2$-cohomology classes] 
{Adjunction for the Grauert--Riemenschneider canonical sheaf and
  extension of \\ $L^2$-cohomology classes}

\thanks{
The first author is supported by the Deutsche Forschungsgemeinschaft (DFG, German Research Foundation), 
grant RU 1474/2 within DFG's Emmy Noether Programme.
The second and third authors are supported by the Swedish Research Council.}

\author{Jean Ruppenthal}
\address{Department of Mathematics, University of Wuppertal, Gau{\ss}str. 20, 42119 Wuppertal, Germany.}
\email{ruppenthal@uni-wuppertal.de}

\author{H{\aa}kan Samuelsson Kalm and Elizabeth Wulcan}
\address{Department of Mathematics, Chalmers University of Technology and the University of G\"oteborg,
S-412 96 G\"oteborg, Sweden.}
\email{hasam@chalmers.se, wulcan@chalmers.se}

\date{\today}

\begin{document}

\begin{abstract}
In the present paper, we derive an adjunction formula for the Grauert-Riemenschneider canonical sheaf
of a singular hypersurface $V$ in a complex manifold $M$.
This adjunction formula is used to study the problem of extending $L^2$-cohomology classes 
of $\dq$-closed forms from the singular hypersurface $V$ to the manifold $M$
in the spirit of the Ohsawa-Takegoshi-Manivel extension theorem.
We do that by showing that our formulation of the $L^2$-extension problem is invariant under
bimeromorphic modifications, so that we can reduce the problem to the smooth case
by use of an embedded resolution of $V$ in $M$.
\end{abstract}

\maketitle


\section{Introduction}

The canonical sheaf of holomorphic forms of top degree is a fundamental object associated
to a complex manifold. It is meaningful not only in the context of Serre duality and various vanishing theorems
of Kodaira type, but plays also an essential role in the classification of complex spaces as it appears
e.g. in the definition of the Kodaira dimension. 

On singular complex varieties, there are different reasonable notions of canonical sheaves
generalizing the canonical sheaf of smooth spaces. On the one hand, there is the dualizing sheaf of Grothendieck
which behaves well with respect to Serre duality and is used for the classification of normal spaces.
On the other hand, there is the canonical sheaf of
Grauert-Riemenschneider which satisfies Kodaira vanishing and behaves very well under bimeromorphic
modifications.

One crucial tool in the study of canonical sheaves of smooth varieties is 
the adjunction formula which gives a link between the canonical sheaf of the ambient space and the canonical sheaf
of a subspace. The adjunction formula sits at the core of e.g. the Ohsawa-Takegoshi-Manivel $L^2$-extension 
theorem, a celebrated analytical result with numerous applications, particularly in algebraic geometry.

However, when it comes to singular spaces, there is a well-known adjunction formula
for the Grothendieck dualizing sheaf, 
but to our knowledge no general adjunction formula is known for the Grauert-Riemenschneider canonical sheaf. 

The main results of the present paper are as follows: we derive an adjunction formula 
for the Grauert-Riemenschneider canonical sheaf
of a singular hypersurface $V$ in a complex manifold $M$,
and we show that our adjunction formula is the 'right' one to be used in the context
of $L^2$-extension problems of Ohsawa-Takegoshi-Manivel type.
In fact, we prove that our formulation of $L^2$-extension problems
is invariant under bimeromorphic modifications.

\medskip
Let us explain this  more precisely.
We consider the following situation: 
let $M$ be a compact complex manifold of dimension $n$ and $V$ a hypersurface
in $M$.
When $V$ is smooth, the well-known adjunction formula reads as
\begin{eqnarray}\label{eq:ad1}
K_V =  \left.\big(K_M \otimes [V] \big)\right|_V = K_M|_V \otimes N_V,
\end{eqnarray}
where $K_V$ and $K_M$ denote the canonical bundles of $V$ and $M$, respectively, $[V]$ is the line bundle associated to $\OO(V)$,
and $N_V=[V]|_V$ is the normal bundle of $V$ in $M$.

In the language of sheaves, \eqref{eq:ad1} can be formulated as the short exact sequence
\begin{eqnarray}\label{eq:ad2}
0 \rightarrow \mathcal{K}_M \hookrightarrow \mathcal{K}_M \otimes \OO(V) \overset{\Psi}{\longrightarrow} \iota_* \mathcal{K}_V \rightarrow 0,
\end{eqnarray}
where $\mathcal{K}_V$ and $\mathcal{K}_M$ are the canonical sheaves of $V$ and $M$, respectively,
and $\iota_* \mathcal{K}_V$ is the trivial extension of $\mathcal{K}_V$ to $M$.
The mapping $\Psi$ is defined by the local equation
\begin{eqnarray}\label{eq:ad3}
\eta = \frac{df}{f} \wedge \Psi(\eta)
\end{eqnarray}
in $\mathcal{K}_M \otimes \OO(V)$ over $V$ where $f$ is any local defining function for the divisor $V$.
More concretely, if $z$ are local holomorphic coordinates on $M$ such that $V=\{z_1=0\}$, then 
$\Psi(g dz_1\wedge \cdots \wedge dz_n/z_1)=g|_V dz_2\wedge \cdots \wedge dz_n$.

When $V$ is a singular hypersurface, the adjunction formula \eqref{eq:ad2} remains valid
if we replace $\mathcal{K}_V$ by Grothendieck's dualizing sheaf $\omega_V = \mathcal{E}\mathit{xt}^{1}_{\OO_M}(\OO_V,\mathcal{K}_M)$,
sometimes also called the Barlet sheaf (see e.g. \cite{PR}, \S 5.3).
However, this adjunction is not compatible with the classical Ohsawa-Takegoshi-Manivel extension theorem
as sections of $\omega_V$ are in general not square-integrable $(n-1)$-forms on the regular part of $V$,
but may have poles of higher order at the singular set $\Sing V$ of $V$.
In contrast, all the classical extension theorems (see \cite{OT}, \cite{Ma}, \cite{De3})
are about extension of holomorphic $L^2$-forms from a possibly singular divisor.

Thus, we propose to replace the canonical sheaf of $V$ in the singular case by
the Grauert-Riemenschneider canonical sheaf of holomorphic square-integrable $(n-1)$-forms
on (the regular part of) $V$. We denote that sheaf again by $\mathcal{K}_V$.

Our first main result is the following adjunction formula for the Grauert-Riemen\-schneider canonical sheaf, cf.\
Section~\ref{ssec:adjunction2}.

\begin{thm}\label{thm:first}
Let $V$ be a (possibly singular) hypersurface in a Hermitian complex manifold $M$.
Then there exists a unique multiplier ideal sheaf $\mathcal{J}(V)$ such that there is
a natural short exact sequence
\begin{eqnarray}\label{eq:ad4}
0 \rightarrow \mathcal{K}_M \hookrightarrow \mathcal{K}_M \otimes \OO(V) \otimes \mathcal{J}(V) \overset{\Psi}{\longrightarrow}
\iota_* \mathcal{K}_V \rightarrow 0,
\end{eqnarray}
where $\mathcal{K}_M$ is the usual canonical sheaf of $M$, $\iota_* \mathcal{K}_V$ is the trivial extension of the Grauert-Riemenschneider
canonical sheaf $\mathcal{K}_V$ of $V$ to $M$, and $\Psi$ is the adjunction map defined by \eqref{eq:ad3} on $V^*=\Reg V$.
The zero set of $\mathcal{J}(V)$ is contained in $\Sing V$.
\end{thm}

We will call the sheaf $\OO(V)\otimes \mathcal{J}(V)$, that now plays
the role of $\OO(V)$, the {\bf
  adjunction sheaf} of $V$ in $M$. In general the adjunction sheaf is
not locally free. 


As a consequence of Theorem \ref{thm:first}, we can deduce:

\begin{thm}\label{thm:gro1}
The natural inclusion $\mathcal{K}_M\otimes\OO(V)\otimes\mathcal{J}(V) \hookrightarrow \mathcal{K}_M\otimes\OO(V)$
induces a natural isomorphism
$$\mathcal{K}_V \cong \omega_V\otimes \mathcal{J}(V)|_V.$$
In particular,
there is a natural inclusion of the Grauert-Riemenschneider canonical sheaf into Grothendieck's dualizing sheaf,
$\mathcal{K}_V \subset \omega_V$.
\end{thm}

If $V$ is normal, the inclusion $\mathcal{K}_V \subset \omega_V$ is
due to Grauert-Riemen\-schneider, see \cite{GrRie}. 
Also, when $V$ is normal we have the following characterization of
when this inclusion is an equality. 


\begin{thm}\label{thm:gro2}
Let $V$ be a normal hypersuface in $M$. Then $\mathcal{K}_V =
\omega_V$, i.e. $\mathcal{J}(V)=\OO_M$, exactly if $V$ has at worst 
canonical singularities.
\end{thm}

\medskip


Taking the long exact sequences of \eqref{eq:ad4} yields 
extension of cohomology classes of $\mathcal K_V$, cf. Remark
~\ref{tuggummi}. 
To obtain our $L^2$-extension results we need to be able to talk
about metrics ``on'' the adjunction sheaf. 
We believe that the following concept is a natural generalization of
smooth metrics on the line bundle $[V]$ associated with $\OO(V)$. 
%
%
%
%
%
%
%
%
%
%
%
If $D$ is an effective divisor in a complex manifold, there is an associated canonical singular
metric $e^{-2\varphi_D}$ on $[D]$; if $D$ is locally defined by
$f_\alpha$, then $\varphi_D$ is locally given as $\log |f_\alpha|$.

\begin{defn}\label{defn:smooth0}
Let $e^{-2\psi}$ be a singular metric on the normal bundle $[V]$ of the hypersurface $V$ in $M$.
We say that $e^{-2\psi}$ is a {\bf smooth metric} on the adjunction sheaf $\OO(V)\otimes\mathcal{J}(V)$ if there exists
an embedded resolution of singularities with only normal crossings
$\pi: (V',M') \rightarrow (V,M)$
such that $e^{-2\pi^*\psi+2\varphi_\Delta}$ is a smooth metric on the normal bundle $[V']$ of $V'$ in $M'$,
where $\Delta=\pi^* V - V'$ is the difference of the total transform $\pi^* V$ and the strict transform $V'$ of $V$.
\end{defn}


In fact, if $\pi: (V',M') \rightarrow (V,M)$ is any embedded resolution
of $V$ in $M$, any smooth metric on $[V']$ induces a smooth metric on
$\OO(V)\otimes\mathcal{J}(V)$, see Theorem ~\ref{thm:smooth}.


For any Hermitian complex space $X$ and a Hermitian line bundle
$F\rightarrow X$ (with a possibly singular metric),
we denote by $\mathcal{C}^{p,q}_X(F)$ the sheaf of germs of
$(p,q)$-forms $g$ on the regular part of $X$ with values in $F$, 
which are square-integrable
up to the singular set and which are in the domain of the $\dq$-operator in the sense of distributions, $\dq_w$,
meaning that $\dq_w g$ is again square-integrable up to the singular
set, see Section ~\ref{sec:l2theory}.
For an open set $U\subset X$, we define the $L^2$-cohomology
$$H^{p,q}_{(2)}(U^*,F) := H^q\big(\Gamma(U,(\mathcal{C}^{p,*}_X(F),\dq_w))\big).$$
Here and in the following,
let $U^*=\Reg U$ denote the regular part of $U$.

We can now give an alternative definition of
$\OO(V)\otimes\mathcal{J}(V)$ as the sheaf of holomorphic sections that are square integrable with
respect to a smooth metric on $\OO(V)\otimes\mathcal{J}(V)$. 



\begin{thm}\label{thm:resolution0}
Let $M$ be a compact complex manifold of dimension $n$ and let $V$ be a
hypersurface in $M$. Moreover let $e^{-2\psi}$ be a metric on $[V]$, which
is smooth on the adjunction sheaf $\OO(V)\otimes\mathcal{J}(V)$.
Then
\begin{equation}\label{waffel}
\OO(V)\otimes \mathcal{J}(V) \cong \ker \dq_w \subset
\mathcal{C}^{0,0}_M([V]^{sing}),
\end{equation} 
where $[V]^{sing}:=([V],e^{-2\psi})$. 

If $F\rightarrow M$ is any Hermitian line bundle (with smooth metric),
then the complex
$\big(\mathcal{C}^{n,*}_M(F\otimes [V]^{sing}),\dq_w\big)$
is a fine resolution for $\mathcal{K}_M(F)\otimes \OO(V)\otimes \mathcal{J}(V)$,
where $\mathcal{K}_M(F)$ denotes the sheaf of holomorphic
$n$-forms with values in $F$. 
\end{thm}

In particular, it follows from Theorem \ref{thm:resolution0} that the (flabby) cohomology of $\mathcal{K}_M(F) \otimes \OO(V) \otimes \mathcal{J}(V)$
can be expressed in terms of the $L^2$-cohomology of forms with values
in the line bundle $F\otimes [V]^{sing}$:
\begin{eqnarray}\label{eq:iso101}
H^q\big(M,\mathcal{K}_M(F)\otimes\OO(V)\otimes\mathcal{J}(V)\big) \cong H^{n,q}_{(2)}\big(M,F\otimes [V]^{sing}\big).
\end{eqnarray}
As $\mathcal{K}_M$ is invertible, Theorem ~\ref{thm:resolution0} also
gives an $L^2$-resolution
for the adjunction sheaf $\OO(V)\otimes\mathcal{J}(V)$. 

The proof of Theorem ~\ref{thm:resolution0} relies on the fact that the metric $e^{-2\psi}$ behaves locally
like $e^{-2\varphi}$ where $\varphi$ is a plurisubharmonic defining function for the multiplier ideal sheaf $\mathcal{J}(V)$.
Using Theorem ~\ref{thm:resolution0}, we will deduce that the adjunction map $\Psi$ from \eqref{eq:ad4} in Theorem ~\ref{thm:first}
induces a natural map $\Psi_V$ on the level of $L^2$-cohomology,
$$\Psi_V: H^{n,q}_{(2)}(M,F\otimes [V]^{sing}) \rightarrow H_{(2)}^{n-1,q}(V^*,F),$$
and thus the extension problem for $L^2$-cohomology classes amounts to
studying the question whether $\Psi_V$ is surjective.

Our second main result shows that
this $L^2$-extension problem is invariant under bimeromorphic
modifications, which allows us to reduce the problem to the case of a
smooth divisor where other methods can be applied.

\begin{thm}\label{kortkort} 
Let $M$ be a compact Hermitian manifold of dimension $n$, let
$V\subset M$ be a singular hypersurface, and let 
$F\rightarrow M$ be a Hermitian holomorphic line bundle. 
Let $e^{-2\psi}$ be a singular Hermitian metric on $[V]$ that is
smooth on the adjunction sheaf $\OO(V)\otimes\mathcal{J}(V)$. 
Moreover, let $$\pi: (V',M') \rightarrow (V,M)$$
be an embedded resolution of singularities of $V$ in $M$.

Then there exists a commutative diagram
\begin{eqnarray}\label{eq:diagram10}
\begin{xy}
  \xymatrix{
H^{n,q}_{(2)}(M,F\otimes [V]^{sing}) \ar[r]^{\ \ \ \ \Psi_V} \ar[d]_\cong^{\pi^*} &
 H^{n-1,q}_{(2)}(V^*,F)
\ar[d]^{\pi^*}_{\cong}\\
H^{n,q}(M',\pi^* F \otimes [V']) \ar[r]^{\ \ \ \Psi_{V'}} &
H^{n-1,q}_{(2)}(V',\pi^*F), }
\end{xy}
\end{eqnarray}
where $\Psi_{V'}$ is the adjunction map for the smooth divisor $V'$ in $M'$,
$\Psi_V$ is induced by the adjunction map for the non-smooth divisor
$V$ in $M$ as defined in \eqref{eq:diagram06}. 
The vertical maps $\pi^*$ are induced by pull-back of forms under $\pi$
and both are isomorphisms.
\end{thm}


Using Theorem \ref{kortkort} together with a recent result by
Berndtsson \cite{B} we get the following $L^2$-extension result
under a quite weak positivity assumption on $F\rightarrow M$.

\begin{thm}\label{thm:extension00}
Let $M$ be a compact  K\"ahler manifold with K\"ahler metric
$\omega$. Furthermore, let $V$, $e^{-2\psi}$, and $F$ be as in Theorem \ref{kortkort}, and let $e^{-2\phi}$ be the smooth metric of the Hermitian line bundle
$F\rightarrow M$. 

Let $0\leq q\leq n-1$. 
Assume that there exists an $\epsilon>0$ such that
\begin{eqnarray}\label{eq:pos1}
i\partial\dq \phi\wedge \omega^q \geq \epsilon i\partial\dq\psi\wedge \omega^q
\end{eqnarray}
and
\begin{eqnarray}\label{eq:pos2}
i\partial\dq \phi\wedge \omega^q\geq 0.
\end{eqnarray}
Then the adjunction map $\Psi_V$ in \eqref{eq:diagram10}
is surjective. 


\end{thm}

Another condition to ensure surjectivity of the adjunction map $\Psi_V$ in \eqref{eq:diagram10}
is $H^{n,q+1}(M,F)=0$. That can be seen by applying the long exact cohomology sequence to the adjunction
formula \eqref{eq:ad4} in Theorem ~\ref{thm:first}, cf. Remark
~\ref{brun}. 
Thus $\Psi_V$ is surjective e.g. if $F\rightarrow M$ is a positive line bundle by Kodaira's vanishing theorem.

In fact, each $L^2$-cohomology classes in $H^{n,q}_{(2)}(M,F
\otimes[V]^{sing})$ has a smooth representative in $\Gamma\big(M,\mathcal{C}^\infty_{n,q}(F)\otimes\OO(V)\otimes\mathcal{J}(V)\big)$, see Lemma
\ref{lem:representative}. 

\begin{cor}\label{cor:extension00}
Let $M$, $V$, and $F$ be as in Theorem \ref{thm:extension00}. 
Assume that the map $\Psi_V$ in \eqref{eq:diagram10} is surjective. 
Let $u\in \Gamma\big(V,\mathcal{C}^{n-1,q}_V(F)\big)$ be a $\dq$-closed $L^2$-form of degree $(n-1,q)$

If $q\geq 1$, then there exists an $L^2$-form $g\in \Gamma\big(V,\mathcal{C}^{n-1,q-1}_V(F)\big)$ and
a $\dq$-closed section $U\in\Gamma\big(M,\mathcal{C}^\infty_{n,q}(F)\otimes \OO(V)\otimes\mathcal{J}(V)\big)$,
i.e., a smooth $(n,q)$-form with values in $F\otimes [V]$ and some extra vanishing according to $\mathcal{J}(V)$,
such that 
$$U = \frac{df}{f}\wedge (u-\dq g)$$
where $f$ is any local defining function for the hypersurface $V$.

If $q=0$ the statement holds without $g$.
\end{cor}

We remark that one can deduce some $L^2$-estimates for our extension by use of Berndtsson's theorem,
but forgo this topic here as these estimates would depend on the resolution of singularities that is used.
Anyway, it is easy to see (and a nice feature) that one can tensor $F$ in such statements with another semi-positive line bundle $F'\rightarrow M$
without changing the estimate.

\medskip 

This paper is organized as follows. In Section ~\ref{sec:l2theory} we
give some background on $L^2$-cohomology on a singular space and the
Grauert--Riemenschneider sheaf. 
In Section \ref{adjsec} we
discuss adjunction and prove Theorem ~\ref{thm:first} in Section
~\ref{ssec:adjunction2} and Theorems ~\ref{thm:gro1} and
~\ref{thm:gro2} in Section ~\ref{grosec}. 
The relation to $L^2$-extension is then discussed in Section 
\ref{sec:l2extension}; the proofs of Theorems ~\ref{thm:resolution0},
~\ref{kortkort}, and ~\ref{thm:extension00} are given in Sections 
\ref{ssec:metrics}, \ref{adjsection}, and \ref{ext_section},
respectively. 
Finally in Section \ref{sec:examples} we illustrate our results by
some examples.

\section{Some preliminaries}\label{sec:l2theory}

\subsection{$L^2$-cohomology for the $\dq$-operator on singular complex spaces}\label{sssec:dqw}


We need to recall briefly some of the concepts from \cite{Rp8}; we
refer to that paper for details.
Let $(X,h)$ always be a (singular) Hermitian complex space of pure dimension $n$,
$F\rightarrow X\setminus \Sing X$ a Hermitian holomorphic line bundle,
and $U\subset X$ an open subset.
Moreover, let $\mathcal L^{*,*}(F)$ be the sheaf of forms with values in $F$ that are
square-integrable up to the singular set of $V$, i.e., 
\begin{eqnarray*}
\mathcal{L}^{p,q}(U,F)
=\{f \in L^{2,loc}_{p,q}(U\setminus\Sing X,F): f|_K \in L^{2}_{p,q}(K\setminus\Sing X,F)\ \forall K\subset\subset U\}.
\end{eqnarray*}
We denote by
$$\dq_w(U): \mathcal L^{p,q}(U,F) \rightarrow \mathcal L^{p,q+1}(U,F)$$
the $\dq$-operator in the sense of distributions on $U\setminus\Sing X$ which is closed and densely defined.
When there is no danger of confusion, we will simply write $\dq$ or $\dq_w$ for $\dq_w(U)$.
Since $\dq_w$ is a local operator, i.e.
$\dq_w(U)|_V = \dq_w(V)$ 
for open sets $V\subset U$, we get sheaves 
$$\mathcal{C}^{p,q}_X(F):=\mathcal{L}^{p,q}(F)\cap \dq_w^{-1}\mathcal{L}^{p,q+1}(F),$$
given by
$$\mathcal{C}^{p,q}_X(U,F) = \mathcal{L}^{p,q}(U,F) \cap\Dom\dq_w(U).$$

The sheaves $\mathcal{C}_X^{p,q}(F)$ admit partitions of unity,
and so we obtain fine sequences
\begin{eqnarray}\label{eq:Cseq1}
\mathcal{C}^{p,0}_X(F) \overset{\dq_w}{\longrightarrow} \mathcal{C}^{p,1}_X(F) \overset{\dq_w}{\longrightarrow} 
\mathcal{C}^{p,2}_X(F) \overset{\dq_w}{\longrightarrow} ...
\end{eqnarray}
We will use simply $\mathcal{C}^{p,q}_X$ to denote the sheaves of forms with values in the trivial line bundle.

The $L^{2,loc}$-Dolbeault cohomology for forms with values in $F$ with respect to the $\dq_w$-operator 
on an open set $U\subset X$
is by definition the cohomology of the complex \eqref{eq:Cseq1} after taking global sections over $U$; 
this is denoted by $H^{p,q}_{(2)}(U^*,F):=H^q(\Gamma(U,\mathcal{C}^{p,*}_X(F)))$.
To be precise, this is the cohomology of forms which are square-integrable on compact subsets of $U$,
and it becomes the $L^2$-cohomology of $X^*$ when $U=X$ is compact.

\subsection{The Grauert--Riemenschneider canonical sheaf}\label{sssec:cgr}

Let $\pi: M\rightarrow X$ be a resolution of singularities and let
$\mathcal{K}_{M}$ be the canonical sheaf of the smooth manifold
$M$. Then the {\bf Grauert--Riemenschneider canonical sheaf} $\mathcal
K_X$ of $X$ is the direct image sheaf 
\begin{eqnarray}\label{babel}
\mathcal{K}_X := \pi_* \mathcal{K}_{M}. 
\end{eqnarray}
The definition is independent of $\pi: M\rightarrow X$, see \cite{GrRie}. 

Since the $L^2$-property of $(n,0)$-forms
remains invariant under modifications of the metric and thus also
under resolution of singularities, $\mathcal K_X$ can alternatively be defined as 
\begin{eqnarray*}
\mathcal{K}_X = \ker \dq_w \subset \mathcal{C}^{n,0}_X. 
\end{eqnarray*}

Let us mention briefly a few important facts which will be used later.
If $F\rightarrow M$ is a semi-positive line bundle and
$\mathcal{K}_M(F)$ is the sheaf of holomorphic $n$-forms
with values in $F$, then Takegoshi's vanishing theorem,
\cite{T}[Remark~2], states that the higher direct images of $\mathcal
K_M$ vanish, 
\begin{eqnarray}\label{taketake}
R^q \pi_* \mathcal{K}_M(F)=0\ ,\ q>0.
\end{eqnarray}

Let both $X$ and $M$ carry Hermitian metrics. Then the pull-back of forms under $\pi$
maps square-integrable $(n,q)$-forms on $X^*$ to square-integrable $(n,q)$-forms on $M$.
By trivial extension over the exceptional set, we get well-defined mappings
\begin{eqnarray*}
\pi^*: \mathcal{C}^{n,q}_X(G) \rightarrow \mathcal{C}^{n,q}_M(\pi^* G) 
\end{eqnarray*}
if $G\rightarrow X$ is any Hermitian holomorphic line bundle.
If the line bundle $\pi^* G \rightarrow M$ is locally semi-positive with respect to the base space $X$,
then this map induces an isomorphism of $L^2$-cohomology, see \cite[Theorem~1.5]{Rp8}:
\begin{equation}\label{jeanthm}
\pi^*: H^{n,q}_{(2)}(X^*,G) 
 \overset{\cong}{\longrightarrow}
H^{n,q}_{(2)}(M,\pi^* G). 
\end{equation}

\section{The adjunction formula}\label{adjsec} 

\subsection{The adjunction formula for a smooth divisor}\label{ssec:adjunction1}

Let $M$ be a complex manifold of dimension $n$, and $V$ a smooth
hypersurface in $M$. 
Let $[V]$ be the holomorphic line bundle whose holomorphic sections correspond to sections in $\OO(V)$; 
by a slight abuse of notation, we call $[V]$ the normal bundle of $V$
in $M$. 
The well-known adjunction formula states that 
\begin{eqnarray}\label{eq:adjunction1}
\iota_* \mathcal{K}_V \cong \mathcal{K}_M \otimes \OO(V) \big/ \mathcal{K}_M = \mathcal{K}_M ([V]) \big/ \mathcal{K}_M.
\end{eqnarray}
Here and throughout $\iota: V \hookrightarrow M$ denotes the natural inclusion,
i.e., $\iota_* \mathcal{K}_V$ is the trivial extension of the canonical sheaf $\mathcal{K}_V$ to $M$.
If we denote the canonical bundles on $V$ and $M$ by $K_V$ and $K_M$, respectively,
then the adjunction formula can be expressed as
\begin{eqnarray*}
K_V \cong \left.\big( K_M\otimes [V]\big)\right|_V.
\end{eqnarray*}

We shall explain how the isomorphism in \eqref{eq:adjunction1} can be realized explicitly.
For further use, we take a slightly more general point of view.
Let $\{f_j\}_j$ be a system of holomorphic functions defining $V$
and let $\omega$ be a smooth $(n,q)$-form with values in $[V]$;
we identify $\omega$ with a semi-meromorphic $(n,q)$-form with at most a single pole along $V$.
In local coordinates $z_1, ..., z_n$, we can write
\begin{eqnarray*}
\omega = \frac{g_j}{f_j} \wedge dz_1\wedge \cdots \wedge dz_n,
\end{eqnarray*}
where the $g_j$ are smooth $(0,q)$-forms which transform as $g_j =
(f_j/f_k)g_k$. 
We can now define the adjunction morphism $\Psi$ 
locally as follows.
For each point $p\in V$, there exists an $f_j$ with $df_j\neq 0$ in a neighborhood of $p$,
i.e. $\partial f_j/\partial z_\mu \neq 0$ for some $1\leq \mu\leq n$. 
In this neighborhood, 
\begin{eqnarray}\label{eq:adjunction2}
\Psi: \left. \omega\mapsto \omega':=(-1)^{q+\mu-1} g_j\wedge 
\frac{dz_1\wedge\cdots \wedge \widehat{dz_\mu}\wedge\cdots \wedge dz_n}{\partial f_j / \partial z_\mu}\right|_V.
\end{eqnarray}
It is not hard to see that this assignement depends neither on the choice of $\mu$ nor on the choice of $f_j$
or the local coordinates $z_1, ..., z_n$.
Thus, \eqref{eq:adjunction2} gives a well-defined mapping
\begin{eqnarray*}
\Psi: C^\infty_{n,q}(M,[V]) \rightarrow C^\infty_{n-1,q}(V)\ .\ 
\end{eqnarray*}
Note that $\Psi$ maps holomorphic $n$-forms with values in $[V]$ on $M$
to holomorphic $(n-1)$-forms on $V$. 
This proves the adjunction formula \eqref{eq:adjunction1}.
Note that on $V$, we have locally:
\begin{eqnarray}\label{eq:adjunction4}
\omega = \frac{df_j}{f_j}\wedge \omega' = \frac{df_j}{f_j}\wedge \Psi(\omega),
\end{eqnarray}
which can be actually used to define $\Psi$.
It follows directly from the definition that $\Psi\circ\dq=\dq\circ\Psi$ so that
the adjunction morphism defines an
adjunction map also on the level of cohomology classes
\begin{eqnarray*}
\Psi_V: H^{n,q}(M,[V]) \rightarrow H^{n-1,q}(V). 
\end{eqnarray*}

\begin{rem}\label{tuggummi}
The question whether $\dq$-cohomology classes on $V$ extend to $M$ or not has a nice cohomological realization.
From the considerations above, we obtain the short exact sequence
\eqref{eq:ad2}. 
By use of the long exact cohomology sequence, it follows that the induced map
$$H^q(M,\mathcal{K}_M\otimes\OO(V)) \longrightarrow  H^q(M,\iota_* \mathcal{K}_V) \cong  H^q(V,\mathcal{K}_V)$$
is surjective if $H^{q+1}(M,\mathcal{K}_M)\cong H^{n,q+1}(M)=0$.
\end{rem}

\subsection{The adjunction formula for a singular divisor}\label{ssec:adjunction2}

Again, let $M$ be a complex manifold of dimension $n$, but $V$ a hypersurface in $M$ which is not necessarily smooth.
Two problems occur. First, it is not clear what we mean by a canonical sheaf on $V$.
Second, the adjunction morphism cannot be defined in a 'smooth' way as above because it will happen that $df_j=0$ 
in singular points of $V$.

Let
$$\pi: (V',M') \rightarrow (V,M)$$
be an embedded resolution of $V$ in $M$, i.e.
$\pi: M' \rightarrow M$
is a surjective proper holomorphic map such that
$\pi|_{M'\setminus E}: M'\setminus E \rightarrow M \setminus \Sing V$
is a biholomorphism, where $E$ is the exceptional divisor which consists of normal crossings only,
and the regular hypersurface $V'$ is the strict transform of $V$, see
e.g. \cite{BiMi}, Theorem 13.2. 
Hence, $\pi|_{V'}: V' \rightarrow V$ 
is a resolution of singularities. For ease of notation, throughout
this paper we will denote $\pi|_{V'}$ again by $\pi$. 

The adjunction formula \eqref{eq:ad2} for the pair $(V',M')$ tells us that
\begin{eqnarray*}
0 \longrightarrow \mathcal{K}_{M'} \hookrightarrow \mathcal{K}_{M'}\otimes\OO(V') \longrightarrow \iota_* \mathcal{K}_{V'} \longrightarrow 0
\end{eqnarray*}
is exact. By Takegoshi's vanishing theorem \eqref{taketake} we get the
short exact sequence 
\begin{eqnarray*}
0 \longrightarrow \pi_* \mathcal{K}_{M'} \hookrightarrow \pi_*\big(\mathcal{K}_{M'}\otimes\OO(V')\big) 
\longrightarrow \pi_* \iota_* \mathcal{K}_{V'} \longrightarrow 0
\end{eqnarray*}
on $M$. In light of \eqref{babel} 
$$
\pi_*\iota_* \mathcal{K}_{V'} = \iota_* \pi_* \mathcal{K}_{V'} \cong
\iota_* \mathcal{K}_V,
$$
where $\iota$ denotes also the embedding $\iota: V \hookrightarrow
M$. 

We thus get the short exact sequence
\begin{eqnarray}\label{eq:exact1b}
0 \longrightarrow \mathcal{K}_{M} \hookrightarrow \pi_*\big(\mathcal{K}_{M'}\otimes\OO(V')\big) 
\longrightarrow \iota_* \mathcal{K}_{V} \longrightarrow 0
\end{eqnarray}
on $M$, and 
obtain the following adjunction formula for the
Grauert-Riemen\-schneider canonical sheaf 
on a singular hypersurface:
\begin{eqnarray}\label{eq:adjunction10}
\iota_* \mathcal{K}_V \cong \pi_* \big( \mathcal{K}_{M'} \otimes \OO(V')\big) \big/ \mathcal{K}_M.
\end{eqnarray}

\medskip
Next, we will show that there exists a  multiplier ideal sheaf 
$\mathcal{J}_\pi(V)$ such that
\begin{eqnarray}\label{eq:mis1}
\pi_* \big( \mathcal{K}_{M'} \otimes \OO(V')\big) \cong \mathcal{K}_M \otimes \OO(V) \otimes \mathcal{J}_\pi(V).
\end{eqnarray}
Since $\pi_*\mathcal{K}_{M'} \cong \mathcal{K}_M$, it is clear that \eqref{eq:mis1} holds with
$\mathcal{J}_\pi(V)_x = \OO_{M,x}$ for points $x\notin \Sing V$.
So, consider a point $x\in\Sing V$. There exists a holomorphic function $f$ in a neighborhood $U_x$ of the point $x$ defining the hypersurface $V$,
i.e. $V=(f)$. Consider the pullback $\pi^* f=f\circ \pi$ on $\pi^{-1}(U_x)$.
Then $\pi^* f$ is vanishing precisely of order $1$ on the strict transform $V'$ of $V$
because $f$ is vanishing precisely of order $1$ on the regular part of $V$.
Let
$$(\pi^* f) = V' + E_f,$$
where $E_f$ is a divisor with support on the exceptional set of the embedded resolution $\pi: M' \rightarrow M$.
$E_f$ is a normal crossing divisor since the exceptional set $E$ has only normal crossings.
$\OO(-E_f)$ is the sheaf of germs of holomorphic functions which vanish at least to the order of $\pi^* f$ on $E$.
For the direct image sheaf, we have that
$$\pi_* \OO(-E_f) \subset \pi_* \OO_{M'} \cong \OO_M$$
on $U_x$, i.e. $\pi_* \OO(-E_f)$ can be considered as a coherent sheaf of ideals in $\OO_M$. So, there exist holomorphic functions $g_1,\ldots, g_k \in \OO(U_x)$ 
that generate the direct image sheaf in a neighborhood of the point $x$.
By restricting $U_x$,
we can assume that this is the case in $U_x$.
Then $\pi^* g_j \in \OO(-E_f)(\pi^{-1}(U_x))$ for $j=1, ..., k$, i.e., all the $\pi^* g_j$ vanish at least to the order of $\pi^* f$ on
the exceptional set $E$, and the common zero set of the $\pi^* g_j$ is contained in $E$.
On the other hand, $f$ is in the direct image sheaf, and so  
$f= h_1 g_1 + ... + h_k g_k$ 
for holomorphic functions $h_1, ..., h_k\in \OO(U_x)$. Hence,
$$\pi^* f = \pi^* h_1 \pi^* g_1 + ... + \pi^* h_k \pi^* g_k,$$
meaning that not all the $\pi^* g_j$ can vanish to an order strictly higher than $\pi^* f$ on any irreducible component of $E$.
Thus, we conclude that
\begin{eqnarray}\label{eq:mis2}
E_f = (\pi^* g_1, ..., \pi^* g_k),
\end{eqnarray}
i.e. the $\pi^*g_1$, ..., $\pi^* g_k$ generate $\OO(-E_f)$.
Let $\varphi:= \log\big( |g_1| + ... + |g_k|\big)$ 
and let $\mathcal{J}(\varphi)$ 
be the multiplier ideal sheaf associated to the plurisubharmonic weight $\varphi$ on $U_x$,
i.e. the sheaf of germs of holomorphic functions $F$ such that $|F|^2 e^{-2\varphi}$ is locally integrable.
On the other hand, 
$\mathcal{J}(\pi^* \varphi)$ 
is the sheaf of germs of holomorphic functions $H$ such that 
$H / (|\pi^* g_1| + ... + |\pi^* g_k|)$ 
is locally square-integrable. But $E$ has only normal crossings, and so \eqref{eq:mis2} implies that
\begin{eqnarray}\label{eq:mis3}
\mathcal{J}(\pi^* \varphi) = \OO(-E_f)
\end{eqnarray}
on $\pi^{-1}(U_x)$. We are now in the position to prove:

\begin{prop}\label{thm:mis1}
On $U_x$, the pull-back of forms under $\pi$ induces the isomorphism
\begin{eqnarray*}
\pi^*:  \mathcal{K}_M \otimes \OO(V) \otimes \mathcal{J}(\varphi) \overset{\cong}{\longrightarrow} \pi_* \big(\mathcal{K}_{M'} \otimes \OO(V')\big).
\end{eqnarray*}
\end{prop}

\begin{proof}
It is well-known that
\begin{eqnarray*}
\pi^*:  \mathcal{K}_M \otimes \mathcal{J}(\varphi) \overset{\cong}{\longrightarrow} 
\pi_* \big( \mathcal{K}_{M'} \otimes \mathcal{J}(\pi^*\varphi)\big),
\end{eqnarray*}
see e.g. \cite{De2}, Proposition 15.5. 
By use of \eqref{eq:mis3}, we obtain
\begin{eqnarray}\label{eq:mis5}
\pi^*: \mathcal{K}_M \otimes \mathcal{J}(\varphi) \overset{\cong}{\longrightarrow} \pi_* \big(\mathcal{K}_{M'} \otimes \OO(-E_f)\big).
\end{eqnarray}
Let $W$ be an open set in $U_x$. Then $\mathcal{K}_M \otimes \mathcal{J}(\varphi)(W)$ are the holomorphic $n$-forms on $W$
with coefficients in $\mathcal{J}(\varphi)$, and $\pi_* \big(\mathcal{K}_{M'} \otimes \OO(-E_f)\big)(W)$
are the holomorphic $n$-forms on $\pi^{-1}(W)$ with coefficients in $\OO(-E_f)$.
The isomorphism in \eqref{eq:mis5} is given by pull-back of $n$-forms under $\pi$.
So, the isomorphism \eqref{eq:mis5} is preserved if we multiply the germs on the left-hand side by $1/f$,
and the germs on the right-hand side by $1/\pi^* f$.
Recalling that $(f)=V$ and $(\pi^* f) = V' + E_f$, it follows that
\begin{eqnarray*}
\pi^*: \mathcal{K}_M \otimes \OO(V) \otimes \mathcal{J}(\varphi) \overset{\cong}{\longrightarrow} 
\pi_*\big(\mathcal{K}_{M'} \otimes \OO(V' + E_f) \otimes \OO(-E_f)\big),
\end{eqnarray*}
meaning nothing else but
\begin{eqnarray*}
\pi^*: \mathcal{K}_M \otimes \OO(V) \otimes \mathcal{J}(\varphi) \overset{\cong}{\longrightarrow} 
\pi_* \big( \mathcal{K}_{M'} \otimes \OO(V')\big).
\end{eqnarray*}
\end{proof}

Note that \eqref{eq:mis5} shows that $\mathcal{J}(\varphi)$ does not depend on 
the specific choice of $\varphi$.
Since $\mathcal{K}_M$ and $\OO(V)$ are invertible sheaves, we conclude the following result which together with
\eqref{eq:exact1b} proves Theorem~\ref{thm:first}.

\begin{thm}\label{thm:mis2}
Let
\begin{eqnarray*}
\mathcal{J}_\pi(V) := \OO(V)^{-1} \otimes \mathcal{K}_M^{-1} \otimes \pi_* \big( \mathcal{K}_{M'} \otimes \OO(V')\big).
\end{eqnarray*}
Then $\mathcal{J}_\pi(V)$ is a multiplier ideal sheaf such that
pull-back of forms under $\pi$ induces the isomorphism
\begin{eqnarray*}
\pi^*: \mathcal{K}_M \otimes \OO(V) \otimes \mathcal{J}_\pi(V) \overset{\cong}{\longrightarrow} \pi_* \big(\mathcal{K}_{M'} \otimes \OO(V')\big).
\end{eqnarray*}
Locally, $\mathcal{J}_\pi(V)=\mathcal{J}(\varphi)$ for a plurisubharmonic function
$\varphi=\log\big(|g_1| + ... + |g_k|\big)$,
where the $g_1, ..., g_k$ are holomorphic functions such that
$$(\pi^* f)|_E = E_f = (\pi^* g_1, ..., \pi^* g_k).$$
\end{thm}

\smallskip

By \eqref{eq:exact1b} and \eqref{eq:adjunction10}, we obtain the short exact sequence
\begin{eqnarray}\label{eq:exact2}
0 \longrightarrow \mathcal{K}_{M} \hookrightarrow \mathcal{K}_M \otimes\OO(V) \otimes\mathcal{J}_\pi(V)
\longrightarrow \iota_* \mathcal{K}_{V} \longrightarrow 0
\end{eqnarray}
and the {\bf adjunction formula for the Grauert-Riemenschneider canonical sheaf
on a singular hypersurface}:
\begin{eqnarray}\label{eq:adjunction100}
\iota_* \mathcal{K}_V \cong \mathcal{K}_M \otimes \OO(V) \otimes \mathcal{J}_\pi(V) \big/ \mathcal{K}_M.
\end{eqnarray}
Since $\mathcal{K}_M$ and $\OO(V)$ are invertible, it follows from \eqref{eq:exact2} that
$\mathcal{J}_\pi(V)$ does not depend on the resolution $\pi$. So, we write $\mathcal{J}(V) = \mathcal{J}_\pi(V)$
and call $\OO(V) \otimes \mathcal{J}(V)$ 
the {\bf adjunction sheaf of $V$ in $M$}. 

Note that the natural injection in \eqref{eq:exact2} makes sense since $\OO_M \subset \OO(V)\otimes \mathcal{J}(V)$,
 $\mathcal{J}(V)$ is coherent,
and the zero set of $\mathcal{J}(V)$ is contained in $\Sing V$.

\begin{rem}\label{brun} 
As in Remark ~\ref{tuggummi} 
we obtain from \eqref{eq:exact2} by use of the long exact cohomology sequence the adjunction map for the flabby cohomology
\begin{eqnarray}\label{eq:flabbyadjunction}
H^q\big(M,\mathcal{K}_M\otimes \OO(V)\otimes \mathcal{J}(V)\big) \longrightarrow  H^q(M,\iota_* \mathcal{K}_V) \cong H^q(V,\mathcal{K}_V),
\end{eqnarray}
which
is surjective e.g. if $H^{q+1}(M,\mathcal{K}_M) \cong H^{n,q+1}(M)=0$.
\end{rem}

\subsection{Relation to the Grothendieck dualizing sheaf}\label{grosec}

As mentioned in the introduction, one can also consider the adjunction formula for
Grothendieck's dualizing sheaf $\omega_V$. 
As $V$ is a complete intersection, 
$$\omega_V \cong \left. \big(\mathcal{K}_M \otimes \OO(V)\big) \right|_V,$$ see, e.g., \cite[\S 5.3]{PR}. 
Thus we obtain the natural short exact sequence
\begin{eqnarray*}
0\rightarrow\mathcal{K}_M \hookrightarrow \mathcal{K}_M\otimes \OO(V)
\rightarrow \iota_* \omega_V \rightarrow 0. 
\end{eqnarray*}
Combining with Theorem ~\ref{thm:mis2} we get the exact commutative diagram
\begin{eqnarray}\label{eq:diagram:gro}
\begin{xy}
  \xymatrix{
      0 \ar[r] & \mathcal{K}_M \ar[r] \ar[d]^{=}    &  \mathcal{K}_M\otimes\OO(V)\otimes\mathcal{J}(V) \ar[r] \ar[d]^{j}& 
\iota_*\mathcal{K}_V \ar[r] \ar@{.>}[d]& 0 \\
      0 \ar[r] & \mathcal{K}_M \ar[r] &  \mathcal{K}_M\otimes\OO(V) \ar[r]  & 
\iota_* \omega_V \ar[r] & 0}
\end{xy}
\end{eqnarray}
where the map $j$ is the natural inclusion. The diagram \eqref{eq:diagram:gro} induces
a natural isomorphism
\begin{eqnarray*}
\iota_*\mathcal K_V\cong 
\big ( \mathcal{K}_M \otimes \OO(V) \otimes \mathcal{J}(V))\big /\mathcal K_M
\cong
\mathcal{J}(V) \otimes \big (\mathcal{K}_M \otimes \OO(V))\big
/\mathcal K_M 
\cong
\mathcal{J}(V)\otimes \iota_* \omega_V.
\end{eqnarray*}
As $\mathcal{J}(V)|_V \subset \OO_V$, this implies particularly that there is a natural
inclusion of the Grauert-Riemenschneider canonical sheaf into Grothendieck's dualizing sheaf, $\mathcal{K}_V \subset \omega_V$.
This proves Theorem ~\ref{thm:gro1}.

\medskip

Let us now prove Theorem ~\ref{thm:gro2}; assume therefore that $V$ is
a normal. 
Then $\omega_V \cong \big(\mathcal{K}_M \otimes \OO(V)\big)|_V$ is an
invertible sheaf corresponding to a canonical Cartier divisor $K_V$,
in particular $V$ is Gorenstein, see e.g. \cite[\S 5.4]{PR}. 
If $\pi: N \rightarrow V$ is any resolution of singularities and $K_N$
is the canonical divisor of $N$,  
we can write
\begin{eqnarray}\label{eq:gro3}
K_N = \pi^* K_V + \sum a_j E_j,
\end{eqnarray}
where the $E_j$ are the irreducible components of the exceptional
divisor and the $a_j$ are rational coefficients. Then $V$ has (at
worst) {\bf
  canonical singularities} if and only if $a_j\geq 0$ for all indices
$j$. 
As \eqref{eq:gro3} is equivalent to $\mathcal{K}_N = \pi^* \omega_V \otimes \OO(\sum a_j E_j)$, 
$V$ has canonical singularities precisely if $\pi^* \omega_V \subset\mathcal{K}_N$. 


\begin{proof}[Proof of Theorem ~\ref{thm:gro2}]
Let $\pi: N\rightarrow V$ be any resolution of singularities.

First assume that $\mathcal{K}_V = \omega_V$. Then, in light of \eqref{babel}, 
$$\pi^* \omega_V = \pi^* \mathcal{K}_V = \pi^*\pi_* \mathcal{K}_N  \subset \mathcal{K}_N,$$
and so $V$ has canonical singularities. 

Conversely, assume that $V$ has canonical
singularities. Then $$\omega_V \subset \pi_* \pi^* \omega_V \subset
\pi_* \mathcal{K}_N = \mathcal{K}_V$$ and thus $\mathcal{K}_V =
\omega_V$, since the inverse inclusion follows
from Theorem ~\ref{thm:gro1}.



\end{proof}

\subsection{The commutative adjunction diagram}\label{ssec:adjunction3}

We will now give an explicit realization of the adjunction map \eqref{eq:flabbyadjunction}
on the level of $(n,q)$-forms. We can achieve that easily by tensoring \eqref{eq:exact2}
with the sheaf of germs of smooth $(0,q)$-forms $\mathcal{C}^\infty_{0,q}$.

Let $M$ be a compact Hermitian manifold, $V$ a singular hypersurface in $M$,
and $\pi: (V',M') \rightarrow (V,M)$ 
an embedded resolution of $V$ in $M$. 
Give $M'$ any positive definite metric.
As in Section ~\ref{ssec:adjunction1}, let $\{f_j'\}_j$ be a finite defining system for $V'$ in $M'$.
We are now in the position to describe an adjunction morphism on $M$ similar to the procedure in Section ~\ref{ssec:adjunction1}.
On $M$, we have to replace the normal bundle $[V]$ by the adjunction sheaf $\OO(V)\otimes\mathcal{J}(V)$.

Consider the adjunction map
$$\Psi': \mathcal{K}_{M'} \otimes \OO(V') \longrightarrow \iota_* \mathcal{K}_{V'}$$
for the non-singular  hypersurface $V'$ in $M'$ defined as in Section
~\ref{ssec:adjunction1}. 
Using Proposition ~\ref{thm:mis1} 
and 
$\pi^*: \mathcal{K}_V \overset{\cong}{\longrightarrow} \pi_*
\mathcal{K}_{V'},$ 
we obtain the first commutative adjunction diagram
\begin{eqnarray}\label{eq:diagram01}
\begin{xy}
  \xymatrix{
  \mathcal{K}_M\otimes\OO(V)\otimes \mathcal{J}(V) \ar[r]^{\ \ \ \ \ \ \ \ \ \ \ \Psi} \ar[d]^{\pi^*}_{\cong} &
 \iota_* \mathcal{K}_V \ar[d]^{\pi^*}_{\cong}\\
 \pi_* \big( \mathcal{K}_{M'} \otimes \OO(V')\big) \ar[r]^{\ \ \ \ \ \ \  \Psi'} &  \pi_* \iota_*\mathcal{K}_{V'} }
\end{xy}
\end{eqnarray}
defining the adjunction map $\Psi$ on $M$. 
Here the vertical maps are induced by pull-back of forms under $\pi$. Note that $\pi_*\iota_* \mathcal{K}_{V'} = \iota_* \pi_* \mathcal{K}_{V'}$
and that the kernel of
$\Psi$ is just the natural inclusion of $\mathcal{K}_M$ in $\mathcal{K}_M\otimes\OO(V)\otimes\mathcal{J}(V)$.
So, $\Psi$
realizes the isomorphism \eqref{eq:adjunction100}, i.e.
$$\Psi: \mathcal{K}_M\otimes\OO(V)\otimes\mathcal{J}(V)\big/ \mathcal{K}_M \overset{\cong}{\longrightarrow} \iota_* \mathcal{K}_V.$$

Considering the maps on global sections, \eqref{eq:diagram01} yields
the commutative diagram
\begin{eqnarray*}
\begin{xy}
  \xymatrix{
  \Gamma\big(M,\mathcal{K}_M\otimes\OO(V)\otimes \mathcal{J}(V)\big) \ar[r]^{\ \ \ \ \ \ \ \ \ \ \ \Psi} \ar[d]^{\pi^*}_{\cong} &
 \Gamma(V,\mathcal{K}_V) \ar[d]^{\pi^*}_{\cong}\\
 \Gamma\big(M',\mathcal{K}_{M'} \otimes \OO(V')\big) \ar[r]^{\ \ \ \ \ \ \ \Psi'} &  \Gamma(V',\mathcal{K}_{V'}) }
\end{xy}
\end{eqnarray*}
We can now describe $\Psi$ explicitly. Let $\omega\in\Gamma(M,\mathcal{K}_M\otimes\OO(V)\otimes\mathcal{J}(V))$.
By use of \eqref{eq:adjunction4}, we get that
\begin{eqnarray}\label{eq:explicit1}
\pi^*\omega = \frac{df_j'}{f_j'} \wedge \Psi'(\pi^*\omega) = \frac{df_j'}{f_j'} \wedge \pi^* \Psi(\omega)
\end{eqnarray}
locally on $V'$.

Let $f$ be a local defining function for $V$ in $M$. Then $\pi^* f$ is vanishing precisely to order $1$
on $V'\setminus E$. Hence
$$\pi^*\left(\frac{df}{f}\right) = \frac{\pi^* df}{\pi^* f} = \frac{df_j'}{f_j'}$$
locally as $(1,0)$-forms in $\OO(V')|_{V'\setminus E}$. Since $\pi$ is a biholomorphism on $M'\setminus E$,
it follows from \eqref{eq:explicit1} that 
\begin{eqnarray}\label{eq:explicit2}
\omega = \frac{df}{f} \wedge \Psi(\omega)
\end{eqnarray}
on $V^*$. This shows that the adjunction map $\Psi$ does not depend on the resolution $\pi: M'\rightarrow M$
since $df$ does not vanish on $V^*$.

\bigskip
Let us now define the adjunction map for (germs of) $(n,q)$-forms.
That can be done simply by tensoring the diagram \eqref{eq:diagram01}
by the sheaves of germs of smooth $(0,q)$-forms, so that we obtain the commutative diagram
\begin{eqnarray}\label{eq:diagram03}
\begin{xy}
  \xymatrix{
  \mathcal{C}^\infty_{0,q} \otimes \big(\mathcal{K}_M\otimes\OO(V)\otimes \mathcal{J}(V)\big) \ar[r]^{\ \ \ \ \ \ \ \ \ \ \ (1,\Psi)} \ar[d]^{(\pi^*,\pi^*)} &
 \mathcal{C}^\infty_{0,q} \otimes \iota_* \mathcal{K}_V \ar[d]^{(\pi^*,\pi^*)}\\
 \pi_* \mathcal{C}^\infty_{0,q} \otimes \pi_* \big( \mathcal{K}_{M'} \otimes \OO(V')\big) \ar[r]^{\ \ \ \ \ (1, \Psi')} &  
\pi_* \mathcal{C}^\infty_{0,q} \otimes \pi_* \iota_*\mathcal{K}_{V'} }.
\end{xy}
\end{eqnarray}
Now the vertical arrows are not isomorphisms any more.

It is easy to see that we can complement the diagram on the right hand side
by natural mappings to the sheaves of germs of $L^2$-forms in the domain of $\dq$ on $V$ and $V'$, respectively:
\begin{eqnarray}\label{eq:diagram04}
\begin{xy}
  \xymatrix{
 \mathcal{C}^\infty_{0,q} \otimes \iota_* \mathcal{K}_V \ar[r] \ar[d]^{(\pi^*,\pi^*)} & 
\iota_* \mathcal{C}^{n-1,q}_V \ar[d]^{\pi^*}\\
\pi_* \mathcal{C}^\infty_{0,q} \otimes \pi_* \iota_*\mathcal{K}_{V'} \ar[r]&
\pi_* \iota_* \mathcal{C}^{n-1,q}_{V'} }
\end{xy}
\end{eqnarray}

Merging \eqref{eq:diagram03}, \eqref{eq:diagram04} and adopting the notation,
we obtain the commutative adjunction diagram for (germs of) $(n,q)$-forms:
\begin{eqnarray*}
\begin{xy}
  \xymatrix{
  \mathcal{C}^\infty_{n,q}\otimes\OO(V)\otimes \mathcal{J}(V) \ar[r]^{\ \ \ \ \ \ \ \ \ \ \ \Psi} \ar[d]^{\pi^*} &
 \iota_* \mathcal{C}^{n-1,q}_V \ar[d]^{\pi^*}\\
 \pi_* \big( \mathcal{C}^\infty_{n,q} \otimes \OO(V')\big) \ar[r]^{\ \ \ \ \ \Psi'} &  
\pi_* \iota_* \mathcal{C}^{n-1,q}_{V'}}
\end{xy}
\end{eqnarray*}
For global sections, we have the commutative diagram
\begin{eqnarray}\label{eq:diagram06}
\begin{xy}
  \xymatrix{
  \Gamma\big(M,\mathcal{C}^\infty_{n,q}\otimes\OO(V)\otimes \mathcal{J}(V)\big) \ar[r]^{\ \ \ \ \ \ \ \ \ \ \ \Psi} \ar[d]^{\pi^*} &
 \Gamma(V,\mathcal{C}^{n-1,q}_V) \ar[d]^{\pi^*}\\
 \Gamma\big(M', \mathcal{C}^\infty_{n,q} \otimes \OO(V')\big) \ar[r]^{\ \ \ \ \ \Psi'} &  
\Gamma(V',\mathcal{C}^{n-1,q}_{V'})}
\end{xy}
\end{eqnarray}

Let $\omega\in\Gamma(M,\mathcal{C}^\infty_{n,q}\otimes\OO(V)\otimes\mathcal{J}(V))$.
It follows from \eqref{eq:explicit2} that we still have
\begin{eqnarray*}
\omega = \frac{df}{f} \wedge \Psi(\omega)
\end{eqnarray*}
locally on $V^*$. 

Since $\dq$ commutes with $\pi^*$ and $\Psi'$, it must also commute with $\Psi$,
and so we deduce from \eqref{eq:diagram06} the commutative diagram on the level of $\dq$-cohomology:
\begin{eqnarray}\label{eq:diagram07}
\begin{xy}
  \xymatrix{
H^q\big(\Gamma(M,\mathcal{C}^\infty_{n,*} \otimes \OO(V)\otimes \mathcal{J}(V))\big) \ar[r]^{\ \ \ \ \ \ \ \ \Psi_V} \ar[d]^{\pi^*} &
 H^q\big(\Gamma(V,\mathcal{C}^{n-1,*}_V)\big)
\ar[d]^{\pi^*}\\
H^q\big(\Gamma(M',\mathcal{C}^\infty_{n,*}\otimes \OO(V'))\big) \ar[r]^{\ \ \ \ \ \Psi_{V'}} &
H^q(\Gamma(V',\mathcal{C}^{n-1,*}_{V'})\big)}
\end{xy}
\end{eqnarray}
Note that the groups on the right hand-side are by definition the $L^2$-cohomology groups
for the $\dq$-operator in the sense of distributions.

It does not cause any additional difficulty to define the adjunction morphisms as above
also for $(n,q)$-forms with values in a Hermitian line bundle $F\rightarrow M$
and in $\pi^* F \rightarrow M'$, respectively. Using \eqref{jeanthm}
we get:

\begin{thm}\label{thm:adjunction1}
Let $M$ be a compact Hermitian manifold of dimension $n$, $V\subset M$ a singular hypersurface in $M$,
and $F\rightarrow M$ a Hermitian line bundle.
Let $\pi: (V',M') \rightarrow (V,M)$
be an embedded resolution of singularities,
and let $\mathcal{J}(V)$ be the multiplier ideal sheaf as defined in
Theorem ~\ref{thm:mis2}. 
Then there exists a commutative diagram
\begin{eqnarray}\label{eq:diagram08}
\begin{xy}
  \xymatrix{
H^q\big(\Gamma(M,\mathcal{C}^\infty_{n,*}(F) \otimes \OO(V)\otimes \mathcal{J}(V))\big) \ar[r]^{\ \ \ \ \ \ \ \ \ \ \ \ \Psi_V} \ar[d]^{\pi^*} &
 H^{n-1,q}_{(2)}(V^*,F)
\ar[d]^{\pi^*}_{\cong}\\
H^q\big(\Gamma(M',\mathcal{C}^\infty_{n,*}(\pi^*F)\otimes \OO(V')\big) \ar[r]^{\ \ \ \ \ \ \ \ \ \Psi_{V'}} &
H^{n-1,q}_{(2)}(V',\pi^*F), }
\end{xy}
\end{eqnarray}
where $\Psi_{V'}$ is the usual adjunction map for the smooth divisor $V'$ in $M'$,
$\Psi_V$ is the adjunction map for the non-smooth divisor $V$ in $M$
as defined in \eqref{eq:diagram07}, and the vertical maps $\pi^*$ in
\eqref{eq:diagram08} are induced by pull-back of forms under $\pi$. 
\end{thm}


Note that the lower line in \eqref{eq:diagram08} can be understood as
\begin{eqnarray*}
\Psi_{V'}: H^{n,q}(M',\pi^* F\otimes [V']) \longrightarrow H^{n-1,q}_{(2)}(V',\pi^* F),
\end{eqnarray*}
where $[V'] \rightarrow M'$ is the normal bundle of $V'$ in $M'$,
see Section ~\ref{ssec:adjunction1}.

Our purpose is to determine conditions under which the adjunction map $\Psi_V$ in the upper line of the commutative diagram \eqref{eq:diagram08}
is surjective. 
For this, it would be interesting to know whether 
the vertical map on the left-hand side of the diagram \eqref{eq:diagram08} is also an isomorphism.
We will see later in Section ~\ref{sec:l2extension} that this is actually true if we replace the upper left corner of the diagram
by a certain $L^2$-cohomology group. In, particular it follows that
vertical map on the left-hand side is surjective, see Corollary
~\ref{cor:surjective}.

\medskip 

Before studying $L^2$-cohomology in the next section, 
let us first investigate the case of $C^\infty$-cohomology on the left
hand side of \eqref{eq:diagram08} a bit closer.

It is well known (just solve the $\dq$-equation locally in the $C^\infty$-category for forms with values
in a line bundle) that the complex
\begin{eqnarray*}
\big( \mathcal{C}^\infty_{n,*}(\pi^* F\otimes [V']),\dq\big)
\end{eqnarray*}
is a fine resolution of $\mathcal{K}_{M'}(\pi^* F\otimes [V'])$. 
On the other hand, we have already seen that pull-back of forms under $\pi$ induces the isomorphism
\begin{equation}\label{bluebird}
\pi^*: \mathcal{K}_M(F)\otimes\mathcal{O}(V)\otimes\mathcal{J}(V) 
\overset{\cong}{\longrightarrow} \pi_* \big( \mathcal{K}_{M'}(\pi^* F) \otimes \OO(V')\big),
\end{equation}
since the line-bundle $F\rightarrow M$ is added easily to the statement of Proposition ~\ref{thm:mis1}.

\begin{lem}\label{lemma:positive}
If $U\subset M$ is sufficiently small, then $\big(\pi^*F\otimes [V']\big)|_{\pi^{-1}(U)}$ is semi-positive.
\end{lem}
\begin{proof}
The factor $\pi^*F$ is irrelevant since the statement is local with respect to $M$ and semi-positivity is stable under pullback by holomorphic mappings. Let $f\in \OO(U)$ be a defining function for $V|_U$ so that

\begin{equation*}
(\pi^*f)=V'|_{\pi^{-1}(U)} + E_f|_{\pi^{-1}(U)},
\end{equation*}
where $E_f$ is a divisor with support on the exceptional set of the resolution.
Then 
\begin{equation*}
[V']|_{\pi^{-1}(U)} = \pi^*[V] \otimes [-E_f]|_{\pi^{-1}(U)}
\end{equation*} 
and so, by the argument above, it is sufficient
to see that $[-E_f]|_{\pi^{-1}(U)}$ is semi-positive.

Take a covering $\{U_{\alpha}\}$ of $\pi^{-1}(U)$ such
that $\pi^*{f}=f_{\alpha}^0\cdot f'_{\alpha}$ in $U_{\alpha}$, where $(f_{\alpha}^0)=V'|_{U_{\alpha}}$
and $(f'_{\alpha})=E_f|_{U_{\alpha}}$. Recall from Section~\ref{ssec:adjunction2} that (possibly after shrinking $U$) there are 
$g_1,\ldots,g_k\in \OO(U)$ that generate the direct image of $\OO(-E_f)$. It follows that there
is a non-vanishing tuple $h_{\alpha}=(h_{1\alpha},\ldots,h_{k\alpha})\in \OO(U_{\alpha})^k$
such that 
\begin{equation*}
\pi^* g=(\pi^* g_1,\ldots,\pi^* g_k)= f'_{\alpha} h_{\alpha} \quad \textrm{in} \quad U_{\alpha}.
\end{equation*} 
Letting $|h_{\alpha}|=(|h_{1\alpha}|^2+\cdots + |h_{k\alpha}|^2)^{1/2}$, it is straight forward to check that
the local functions $e^{-2\log|h_{\alpha}|}$ transform as a metric on $[-E_f]$. Since $\log|h_{\alpha}|$ is 
plurisubharmonic, $[-E_f]|_{\pi^{-1}(U)}$ is semi-positive. 
\end{proof}

Thus, $\pi^* F \otimes [V']$ is locally semi-positive with respect to $M$. 
Since $\mathcal{K}_{M'}(\pi^*F \otimes [V'])\cong \mathcal{K}_{M'}(\pi^*F)\otimes \OO(V')$
we conclude by Takegoshi's vanishing theorem \eqref{taketake}:
\begin{equation}\label{dino} 
R^q\pi_* \big(\mathcal{K}_{M'}(\pi^* F) \otimes \OO(V')\big)=0
\ \ \ \ \mbox{ for } \ \ q>0.
\end{equation} 
So, the direct image complex
\begin{eqnarray*}
\left(\pi_* \big(\mathcal{C}^\infty_{n,*}(\pi^* F) \otimes\OO(V')\big), \pi_* \dq\right)
\end{eqnarray*}
is a fine resolution of $\pi_* \big( \mathcal{K}_{M'}(\pi^* F) \otimes
\OO(V')\big)$
%
Combining with \eqref{bluebird} and taking global cohomology 
we would get that the vertical map on the left-hand side of \eqref{eq:diagram08}
is an isomorphism if we knew that the complex
$$\big(\mathcal{C}^\infty_{n,*}(F)\otimes \OO(V)\otimes\mathcal{J}(V),\dq\big)$$
were exact.
But this is a delicate problem since it involves the multiplier ideal
sheaf $\mathcal{J}(V)$. We can prove the required local
exactness only in the $L^2$-category, see Section
~\ref{sec:l2extension}.


\section{Extension of cohomology classes}
\label{sec:l2extension}

\subsection{Smooth metrics on the adjunction sheaf  $\OO(V)\otimes\mathcal{J}(V)$}
\label{ssec:metrics}

Given a line bundle $F\rightarrow M$ we will use the notation
$e^{-2\varphi}$ for the Hermitian metric on $F$ with weight
$\varphi$. When $e^{-2\varphi}$ is a singular Hermitian metric, i.e.
the weight $\varphi$ is in $L^{1,loc}$ rather than smooth,
cf. \cite[Definition~11.20]{De2}, we will sometimes write $F^{sing}$. 
The following result asserts that smooth metrics on adjunction
sheaves, as defined in Definition ~\ref{defn:smooth0} do
actually exist. 

\begin{thm}\label{thm:smooth}
Let
$$\pi: (V',M') \rightarrow (V,M)$$
be any embedded resolution of singularities of $V$ in $M$ with only normal crossings,
and $e^{-2\psi'}$ a smooth Hermitian metric on the normal bundle $[V']$ of $V'$ in $M'$.
Then $e^{-2\psi'}$ induces a singular metric $e^{-2\psi}$ on $[V]$ that is smooth
on the adjunction sheaf $\OO(V)\otimes\mathcal{J}(V)$.
Moreover, if $[V']|_{\pi^{-1}(U)}$ is semi-positive, then $[V]^{sing}|_{U}$ is too. 
\end{thm}

\begin{proof}
It is sufficient to prove the theorem locally in $M$ so let $U\subset M$ such that $V|_U$ 
is defined by $f\in \OO(U)$.
Then $\pi^* V= (\pi^* f) = V'|_{\pi^{-1}(U)} + E_f|_{\pi^{-1}(U)}$,
where $V'$ is the strict transform of $V$, and $E_f$ is an effective divisor with support on the exceptional set $E$;
cf.\ the proof of Lemma~\ref{lemma:positive}.
Take a covering $\{U_{\alpha}\}$ of $\pi^{-1}(U)$ such that
$\pi^*f=f_{\alpha}^0 f_{\alpha}'$ in $U_{\alpha}$ where $f_{\alpha}^0$ defines $V'|_{U_{\alpha}}$ 
and $f_{\alpha}'$ defines $E_f|_{U_{\alpha}}$.
Now, if $s$ is any section of $[V]$ over $U$ then $\pi^*s/f_{\alpha}'$ transforms
as the $f_{\alpha}^0$ and hence defines a semi-meromorphic section of $[V']$ over $U_{\alpha}$. 
We define the singular metric on $[V]$ by letting
\begin{equation}\label{eq:metrik}
|s|_V^2:= |\pi^*s|^2/|f_{\alpha}'|^2e^{-2\psi'_{\alpha}}=
|\pi^*s|^2  e^{-2(\psi'_{\alpha}+\varphi_{E_f})},
\end{equation}
where $e^{-2\psi'_{\alpha}}$ are the local functions for the metric on $[V']$.
Clearly, this singular metric on $[V]$ induces the original metric $e^{-2\psi'}$ on $[V']$ and hence,
the singular metric on $[V]$ is smooth on the adjunction sheaf $\OO(V)\otimes \mathcal{J}(V)$.

Assume now that $[V']|_{\pi^{-1}(U)}$ is semi-positive, i.e., that $dd^c \psi'_{\alpha}\geq 0$;
notice that by Lemma~\ref{lemma:positive} this can always
be achieved if $U$ is small enough. By \eqref{eq:metrik}, the singular metric $e^{-2\psi}$ on $[V]$ has the property
that $\pi^*\psi=\psi'_{\alpha}+\log|f'_{\alpha}|$ and so, in the sense of currents,
\begin{equation*}
dd^c \psi = \pi_* (dd^c \psi'_{\alpha} + T)=\pi_* dd^c \psi'_{\alpha},
\end{equation*}
where $T$ is the current of integration on $|E_f|$; the last equality follows since $\pi_*T$ is a normal
$(1,1)$-current with support on $\textrm{Sing}\, V$, which has codimension $\geq 2$.
Hence, $[V]^{sing}|_U$ is semi-positive.
\end{proof}

Next, we will prove Theorem ~\ref{thm:resolution0}. 
Let $e^{-2\psi}$ be the induced metric on $[V]$ from Theorem~\ref{thm:smooth}.
With the notation from the proof above, we then have that $\pi^*e^{-2\psi} \sim |f_{\alpha}'|^{-2}$.
Recall from Section ~\ref{ssec:adjunction2} that we can choose holomorphic functions $g_1, ..., g_k\in\OO(U)$
such that $g_1, ...,g_k$ generate the direct image of $\OO(-E_f)$ over $U$,
where $U$ is a small open set in $M$,
i.e. $E_f$ is precisely the common zero set of $\pi^* g_1$, ..., $\pi^* g_k$ (counted with multiplicities).
Hence, $|f_{\alpha}'|\sim |\pi^* g_1|+\cdots + |\pi^* g_k|$,
and so
\begin{equation}\label{eq:psi10}
e^{-2\psi} \sim \big( |g_1|+ ... + |g_k|\big)^{-2}=e^{-2\varphi},
\end{equation}
where $\varphi$ is a local defining function for the multiplier ideal sheaf $\mathcal{J}(V)$.


\begin{rem}\label{ifon} 
Let $e^{-2\psi}$ and $e^{-2\psi'}$ be smooth metrics on $\mathcal
O(V)\otimes \mathcal J (V)$. Then note, in light of \eqref{eq:psi10},
that being $L^2$ with respect to $e^{-2\psi}$ is equivalent to being
$L^2$ with respect to $e^{-2\psi'}$. 
\end{rem}


\begin{proof}[Proof of Theorem \ref{thm:resolution0}] 
By Definition ~\ref{defn:smooth0}, there exists an embedded resolution 
$\pi: (V',M')\rightarrow (V,M)$ 
of $V$ in $M$ such that $\pi^* e^{-2\psi}$ induces a smooth metric on the normal bundle $[V']$
of $V'$ in $M'$. But then, by \eqref{eq:psi10}, $\psi\sim\varphi$
where $\varphi$ is a local defining function for the multiplier ideal sheaf $\mathcal{J}(V)$.
Hence, a holomorphic section $h$ of $[V]$ is in $\mathcal{C}^{0,0}_M([V]^{sing})$ precisely if
$|h|^2 e^{-2\varphi}$ is locally integrable (in a trivialization of
$[V]$). This proves \eqref{waffel}.

\smallskip 

It follows from \eqref{waffel} that 
sections of $\mathcal{K}_M\otimes\OO(V)\otimes\mathcal{J}(V)$
can be identified with square-integrable holomorphic sections of $[V]^{sing}$, i.e.
\begin{eqnarray*}
\mathcal{K}_M \otimes \OO(V) \otimes \mathcal{J}(V) \cong \ker \dq_w \subset \mathcal{C}^{n,0}_M([V]^{sing}).
\end{eqnarray*}
By Remark ~\ref{ifon} we may assume that
$e^{-2\psi}$ is locally semi-positive.
Then  
exactness of the complex $\big(\mathcal{C}^{n,*}_M([V]^{sing}),\dq_w\big)$
is equivalent to local $L^2$-exactness of the $\dq$-equation for $(n,q)$-forms with values in a
holomorphic line bundle with a singular Hermitian metric which is
positive semi-definite. 
But this is well-known, see, e.g., \cite[Corollary~14.3]{De2}, and the proof of the Nadel vanishing theorem,
\cite[Theorem~15.8]{De2}. 
It is furthermore clear that the sheaves $\mathcal{C}^{n,q}_M([V]^{sing})$ admit a smooth partition of unity,
so that $\big(\mathcal{C}^{n,*}_M([V]^{sing}),\dq_w\big)$ is in fact a fine resolution of $\mathcal{K}_M \otimes \OO(V) \otimes \mathcal{J}(V)$.

\end{proof}

\subsection{The adjunction diagram for $L^2$-cohomology
  classes}\label{adjsection}

Let $[V]^{sing}$ be the normal bundle of a hypersurface $V$ with a singular Hermitian metric, smooth on $\OO(V)\otimes\mathcal{J}(V)$.
In order to define the adjunction map for $L^2$-cohomology classes with values in $[V]^{sing}$
we need the following lemma. This is necessary as $L^2$-forms do not behave well
under restriction to lower-dimensional subspaces.

\begin{lem}\label{lem:representative}
Each $L^2$-cohomology class $[\phi]\in H^{n,q}_{(2)}(M,[V]^{sing})$ has a smooth representative
$\phi \in\Gamma\big(M,\mathcal{C}^\infty_{n,q}\otimes \OO(V) \otimes\mathcal{J}(V)\big)$.
\end{lem}

\begin{proof}
We will use the DeRham-Weil-Dolbeault isomorphism (see e.g. Demailly, \cite[IV.6]{De4},
or adopt the procedure from \cite[Chapter~7.4]{Hoe3}). 
Let $\mathcal{U}=\{U_\alpha\}$ be a Stein cover for $M$ and $\{\chi_\alpha\}$
a smooth partition of unity subordinate to $\mathcal{U}$.
Recall that the DeRham-Weil-Dolbeault map on \v{C}ech cohomology
\begin{eqnarray}\label{eq:RWD1}
[\Lambda_q]: \check{H}^q\big(\mathcal{U},\mathcal{K}_M\otimes \OO(V) \otimes\mathcal{J}(V)\big)
\longrightarrow H^{n,q}_{(2)}(M,[V]^{sing})
\end{eqnarray}
is defined as follows: given a \v{C}ech cocycle $c\in C^q(\mathcal{U},\mathcal{K}_M\otimes \OO(V) \otimes\mathcal{J}(V))$,
set
\begin{eqnarray}\label{eq:RWD2}
\Lambda_q c :=\sum_{\nu_0, ..., \nu_q} c_{\nu_0\cdots\nu_q} \chi_{\nu_q} \dq \chi_{\nu_0}\wedge\cdots \wedge\dq\chi_{\nu_{q-1}}.
\end{eqnarray}
As $(\mathcal{C}^{n,*}_M([V]^{sing}),\dq_w)$ is a fine resolution for $\mathcal{K}_M\otimes\OO(V)\otimes\mathcal{J}(V)$
(see Theorem ~\ref{thm:resolution0} and \eqref{eq:iso101}), $[\Lambda_q]$ is an isomorphism.
So, each class $[\phi]\in H^{n,q}_{(2)}(M,[V]^{sing})$ has a representative
$$\phi=\Lambda_q c \in \Gamma\big(M,\mathcal{C}^\infty_{n,q}\otimes\OO(V)\otimes\mathcal{J}(V)\big).$$
\end{proof}

We are now ready to prove Theorem \ref{kortkort}, 
replacing the upper left cohomology group in \eqref{eq:diagram08} by $H^{n,q}_{(2)}(M,[V]^{sing})$.
Note that it does not cause any difficulty to include a Hermitian holomorphic line bundle $F\rightarrow M$
in the statement of Theorem ~\ref{thm:resolution0} and Lemma ~\ref{lem:representative}.

\begin{proof}[Proof of Theorem \ref{kortkort}] 
Starting from Proposition ~\ref{thm:adjunction1}, we replace the cohomology group
$$H^q\big(\Gamma(M,\mathcal{C}^\infty_{n,*}(F)\otimes\OO(V)\otimes\mathcal{J}(V))\big)$$
in the upper left corner of the commutative diagram \eqref{eq:diagram08} by the $L^2$-cohomology
$$H^{n,q}_{(2)}(M,F\otimes [V]^{sing}) = H^q\big(\Gamma(M,\mathcal{C}^{n,*}_M(F\otimes [V]^{sing}))\big).$$
We do that by adding the map
$$\Lambda_q\circ [\Lambda_q]^{-1}: H^{n,q}_{(2)}(M,F\otimes [V]^{sing}) 
\longrightarrow H^q\big(\Gamma(M,\mathcal{C}^\infty_{n,*}(F)\otimes\OO(V)\otimes\mathcal{J}(V))\big)$$
to the diagram \eqref{eq:diagram08}, 
where $\Lambda_q$ is the DeRham-Weil-Dolbeault map as defined in \eqref{eq:RWD1}, \eqref{eq:RWD2}.
The application of $\Lambda_q\circ [\Lambda_q]^{-1}$ means to choose smooth
representatives in $\Gamma(M,\mathcal{C}^\infty_{n,q}(F)\otimes\OO(V)\otimes\mathcal{J}(V))$
for cohomology classes in $H^{n,q}_{(2)}(M,F\otimes [V]^{sing})$.
Note that $\Lambda_q\circ[\Lambda_q]^{-1}$ does not depend on the choices made in Lemma ~\ref{lem:representative}
as we consider the map on cohomology classes.

By use of Proposition ~\ref{thm:adjunction1}, it only remains to show that
\begin{eqnarray*}
\pi^*\circ \Lambda_q\circ [\Lambda_q]^{-1}: H^{n,q}_{(2)}(M,F\otimes [V]^{sing}) \longrightarrow H^{n,q}(M',\pi^* F\otimes [V'])
\end{eqnarray*}
is an isomorphism. This is equivalent to showing that
\begin{eqnarray*}
\pi^*\circ \Lambda_q: \check{H}^q(\mathcal{U},\mathcal{K}_M(F)\otimes\OO(V)\otimes\mathcal{J}(V))
\rightarrow H^{n,q}(M',\pi^* F\otimes [V'])
\end{eqnarray*}
is an isomorphism, where $\mathcal{U}=\{U_\alpha\}$ is a Stein cover for $M$
and $\Lambda_q$ is the DeRham-Weil-Dolbeault map with respect to a suitable partition of unity $\{\chi_\alpha\}$
subordinate to $\mathcal{U}$.
But $\pi^* \circ \Lambda_q= \Lambda_q' \circ \pi^*$,
where we let $\Lambda_q'$ denote the DeRham-Weil-Dolbeault map with respect to the covering $\pi^*\mathcal{U}=\{\pi^{-1}(U_\alpha)\}$
and the partition of unity $\{\pi^* \chi_\alpha\}$ on $M'$.
Using \eqref{bluebird} we get 
\begin{eqnarray*}
\pi^*: \check{H}^q(\mathcal{U},\mathcal{K}_M(F)\otimes\OO(V)\otimes\mathcal{J}(V))
\overset{\cong}{\longrightarrow}
\check{H}^q(\pi^*\mathcal{U},\mathcal{K}_{M'}(\pi^* F) \otimes
\OO(V')) 
\end{eqnarray*} 
and thus by \eqref{dino}, 
$$H^q\big(\pi^{-1}(U),\mathcal{K}_{M'}(\pi^* F) \otimes \OO(V')\big)
= H^q\big(U,\mathcal{K}_M(F)\otimes\OO(V)\otimes\mathcal{J}(V)\big)$$
on open sets $U\subset M$. Hence, $\pi^* \mathcal{U}$ is a Leray cover for $\mathcal{K}_{M'}(\pi^* F) \otimes \OO(V')$
on $M'$, meaning that the DeRham-Weil-Dolbeault map $\Lambda_q'$ on $M'$ is also an isomorphism.
Hence, $\pi^*\circ\Lambda_q = \Lambda_q'\circ\pi^*$ is an isomorphism.
\end{proof}

The proof of Theorem ~\ref{kortkort} yields immediately:

\begin{cor}\label{cor:surjective}
The vertical map $\pi^*$ on the left-hand side of the commutative diagram \eqref{eq:diagram08}
in Proposition ~\ref{thm:adjunction1} is surjective.
\end{cor}

\subsection{Extension of $L^2$-cohomology classes}\label{ext_section}






\begin{proof}[Proof of Theorem \ref{thm:extension00}] 
By Theorem ~\ref{kortkort}, we can consider instead the extension problem for the smooth hypersurface $V'$ in $M'$. 
So, we have to discuss the question whether
\begin{eqnarray*}
\Psi_{V'}: H^{n,q}(M',\pi^* F \otimes [V']) \longrightarrow H^{n-1,q}_{(2)} (V', \pi^* F)
\end{eqnarray*}
is surjective, where $\pi^* F$ carries the smooth Hermitian metric $\pi^* e^{-2\phi}$ and
$[V']$ carries the smooth Hermitian metric $e^{-2\psi'}$
where $\psi'=\pi^* \psi - \varphi_\Delta$ with $\Delta=\pi^* V -V'$.
Recall that locally $\varphi_\Delta=\log |f'_\alpha|$ (see the proof of Theorem \ref{thm:smooth}),
so that
\begin{eqnarray}\label{eq:pos2b}
i\partial\dq \psi' = i\partial\dq \pi^*\psi - i\partial\dq \log|f'_\alpha| \leq i\partial \dq \pi^*\psi.
\end{eqnarray}
As $\pi: M' \rightarrow M$ is holomorphic, \eqref{eq:pos1}, \eqref{eq:pos2} and \eqref{eq:pos2b} give
\begin{eqnarray}\label{eq:pos3}
\big(i\partial\dq \pi^* \phi - \epsilon i\partial\dq \psi' \big) \wedge (\pi^* \omega)^q \geq
\big(i\partial\dq \pi^* \phi - \epsilon i\partial\dq \pi^* \psi\big) \wedge (\pi^* \omega)^q &\geq& 0\ ,\\
i\partial\dq \pi^* \phi \wedge (\pi^*\omega)^q &\geq& 0.\label{eq:pos4}
\end{eqnarray}

We may assume that the embedded resolution of $V$ in $M$
is obtained by finitely many blow-ups (i.e. monoidal transformations) along smooth centers, see \cite[Theorem~13.4]{BiMi}.
So, $M'$ can be interpreted as a submanifold in a finite product of K\"ahler manifolds
and it inherits a K\"ahler metric $\omega'$.

As $\omega'$ is strictly positive definite and $\pi^* \omega$ is only positive semi-definite,
there exists a constant $C>0$ such that $\omega' \geq C\pi^* \omega$.
Thus, \eqref{eq:pos3} and \eqref{eq:pos4} imply
\begin{eqnarray*}
\big(i\partial\dq \pi^* \phi - \epsilon i\partial\dq \psi' \big) \wedge (\omega')^q &\geq& 
C^q \big(i\partial\dq \pi^* \phi - \epsilon i\partial\dq \psi' \big) \wedge (\pi^* \omega)^q \geq  0\ ,\\
i\partial\dq \pi^* \phi \wedge (\omega')^q &\geq& C^q i\partial\dq \pi^* \phi \wedge (\pi^* \omega)^q \geq 0.
\end{eqnarray*}
So, the smooth metrics $\pi^* e^{-2\phi}$ and $e^{-2\psi'}$
satisfy the assumptions of Berndtsson's extension Theorem 3.1 in
\cite{B} 
for the smooth divisor $V'$ in the K\"ahler manifold $(M',\omega')$,
and this gives the required surjectivity of $\Psi_{V'}$.
\end{proof}


Combining Theorem ~\ref{thm:extension00} with Lemma ~\ref{lem:representative}, we obtain also Corollary~\ref{cor:extension00}.


%

\section{Examples for the multiplier ideal sheaf $\mathcal{J}(V)$}\label{sec:examples}

We shall illustrate the role of the multiplier ideal sheaf $\mathcal{J}(V)$
and of our adjunction sheaf $\OO(V)\otimes\mathcal{J}(V)$, respectively, in three simple examples.

\begin{example} 
Let us discuss briefly what would happen if we blew up a regular hypersurface.
So, let $V$ be the regular hypersurface in $\C^2$ (with coordinates $z_1, z_2$)
given as the zero set of $f(z)=z_1$.
Let $\pi: M' \rightarrow \C^2$ be the blow up of the origin, i.e. $M'$ is given by the equation
$z_1w_2=z_2w_1$ in $\C^2\times \C\mathbb{P}^1$ with coordinates
$((z_1, z_2); [w_1:w_2])$ and $\pi$ is the projection $\C^2\times
\C\mathbb{P}^1\to \C^2, (z,w)\mapsto z$.

We cover $M'$ by two charts.
The first is given by $w_1=1$ (coordinates $z_1, w_2$).
Here, $\pi^*f = z_1$ and the exceptional divisor $E$ appears as $\{z_1=0\}$.

The second chart is given by $w_2=1$ (coordinates $w_1, z_2$). Here, $\pi^* f=z_2 w_1$,
the exceptional divisor $E$ appears as $\{z_2=0\}$, and the strict transform $V'$ of $V$
is just $\{w_1=0\}$.

Thus, in the notation of Section ~\ref{ssec:adjunction2}, we have $E_f=E$ so that $E_f$
is generated by the two holomorphic functions $\pi^*g_1$ and $\pi^* g_2$ where $g_1=z_1$ and $g_2=z_2$.
So, $\mathcal{J}(V)=\mathcal{J}(\varphi)$ with $\varphi=\log\big(|z_1| + |z_2|\big)$.
Let $h\in (\OO_{\C^2})_0$ be a germ of a holomorphic function at the origin of $\C^2$.
Then $he^{-\varphi}=h/(|z_1|+|z_2|)$ is locally square-integrable at the origin.
We conclude that $\mathcal{J}(V)=\mathcal{J}(\varphi)=\OO_{\C^2}$,
which is expected since $V$ is smooth. 

\end{example}


\begin{example} 
Let $V$ be the cusp in $\C^2$ (with coordinates $z_1,z_2$) given as the zero set of $f(z)=z_1^3-z_2^2$.
We obtain an embedded resolution $\pi: (V',M') \rightarrow (V,\C^2)$ with only normal crossings by
a sequence of three blow-ups. $M'$ can be realized as follows. We consider $\C^2\times \C\mathbb{P}^1\times \C\mathbb{P}^1\times \C\mathbb{P}^1$
with coordinates $\big((z_1,z_2); [w_1:w_2]; [x_1:x_2]; [y_1:y_2] \big)$ and define $M'$ by the three
equations
\begin{equation*}
z_1w_2  =  z_2w_1, ~~~~ 
z_1x_2  =  w_2 x_1, ~~~~ 
x_1y_2  =  w_2 y_1.
\end{equation*}
The resolution $\pi$ is given by the projection on the first factor $\C^2$.
The exceptional set consists of three copies $E_1, E_2, E_3$ of $\C\mathbb{P}^1$ coming from the three blow-ups.
It is not hard to check that $(\pi^*f) = V' + 2E_1 + 3E_2 + 6E_3$. 
The whole resolution $M'$ can be covered by eight charts,
but we can get a good picture by just considering two of them.

The first is given by $w_1=x_2=y_1=1$ (coordinates $x_1, y_2$). Then $z_1=x_1^2y_2$ and $z_2=x_1^3y_2^2$
so that 
$\pi^* f= x_1^6 y_2^3 (1-y_2).$ 
Here, $E_2=(y_2)$, $E_3=(x_1)$ and $V'=\{1-y_2=0\}$.

The second interesting chart is given by $w_1=x_2=y_2=1$ (coordinates $y_1, w_2$).
Then $z_1=y_1w_2^2$ and $z_2=y_1 w_2^3$ so that 
$\pi^* f= w_2^6 y_1^2(y_1-1).$ 
Here, $E_1=(y_1)$, $E_3=(w_2)$ and $V'=\{y_1-1=0\}$.

One can check that 
$E_f=2E_1 + 3E_2 + 6E_3$ 
is generated by the two holomorphic functions $\pi^* g_1$ and $\pi^* g_2$ where $g_1=z_1^3$ and $g_2=z_2^2$.
So, we obtain here that $\mathcal{J}(V)=\mathcal{J}(\varphi)$ with $\varphi=\log\big(|z_1|^3+|z_2|^2\big)$.
It is a standard exercise to compute the multiplier ideal sheaf $\mathcal{J}(\varphi)$ (see e.g. \cite{De2}, Exercise 15.7):
$$\mathcal{J}(V)=\mathcal{J}(\varphi) = (z_1,z_2),$$
i.e. we obtain the ideal sheaf of the origin (with multiplicity one).

It is now interesting to check that this makes sense in view of the adjunction mapping
$$\Psi: \mathcal{K}_{\C^2} \otimes \OO(V) \otimes \mathcal{J}(V) \rightarrow \mathcal{K}_V.$$
A germ $\omega$ of $\mathcal{K}_{\C^2} \otimes \OO(V) \otimes \mathcal{J}(V)$ can be written as
$$\omega = g \frac{dz_1\wedge dz_2}{f} = g \frac{dz_1\wedge dz_2}{z_1^3-z_2^2},$$
where $g$ is a germ of a holomorphic function in $\mathcal{J}(V)$.
As a cusp, $V$ has a well-defined tangential space at the origin, that is $T_0V=\{z_2=0\}$.
Thus, $|dz_1|_V\sim 1$ in a neighborhood of the origin on $V$ (measured in the metric on $V$ induced by the Euclidean metric of $\C^2$).
Recall that
$$\Psi(\omega)= - g \frac{dz_1}{\partial f/\partial z_2} = g \frac{dz_1}{2z_2}.$$
Recall that by definition $\Psi (\omega)$ is in the
Grauert-Riemenschneider canonical sheaf $\mathcal{K}_V$ if and only if
it is square-integrable on $V$. But
$$|\Psi(\omega)|_V = \left| g \frac{dz_1}{2 z_2}\right|_V \sim |g| \frac{1}{|z_1|^{3/2}},$$
which would not be square-integrable on $V$ at the origin if $g$ were just a holomorphic function.
But $g$ is of the form $g=z_1 h_1$ or $g=z_2 h_2$ so that $\Psi(\omega)$ is in fact square-integrable on $V$.
This illustrates the role of the multiplier ideal sheaf $\mathcal{J}(V)$ in the adjunction formula.\\
\end{example}

\begin{example} 
Consider the hypersurface  $V$ generated by $f(x,y,z)=z^2-xy$ in $\C^3$ (with coordinates $x,y,z$).
The embedded resolution $\pi: (V',M') \rightarrow (V,M)$ is obtained by a single blow-up of the origin.
The exceptional set $E$ is a single copy of $\C\mathbb{P}^2$ and it easy to check that $E_f=2E$
as $f$ vanishes to order $2$ in the origin.

One can check that $E_f=2E$ is generated by the three holomorphic functions $\pi^* g_1$, $\pi^* g_2$ and $\pi^* g_3$
where $g_1=x^2$, $g_2=y^2$ and $g_3=z^2$. It follows that $\mathcal{J}(V)=\mathcal{J}(\varphi)$
with 
$\varphi=\log\big( |x|^2 + |y|^2 + |z|^2\big).$ 
But if $h$ is a holomorphic function, then $he^{-\varphi}= h/(|x|^2+|y|^2+|z|^2)$ is square-integrable in $\C^3$.
Thus, 
$\mathcal{J}(V)=\mathcal{J}(\varphi)=\OO_{\C^3},$ 
i.e. $V$ has a canonical singularity (see Theorem ~\ref{thm:gro2}).

We shall check that the adjunction map
$$\Psi: \mathcal{K}_{\C^3} \otimes\OO(V) \otimes \mathcal{J}(V) = \mathcal{K}_{\C^3}\otimes \OO(V) \longrightarrow \mathcal{K}_V$$
makes sense. A section $\omega$ of $\mathcal{K}_{\C^3}\otimes\OO(V)$ can be written as
$$\omega= g \frac{dx\wedge dy \wedge dz}{z^2-xy},$$
where $g$ is just a holomorphic function. To check that $\Psi(\omega)$ is square-integrable on $V$, we cover the variety $V$
with two charts. For that purpose we cover $V$ by the two parts where either $|x|\geq |y|$ or $|y|\geq |x|$, respectively.

Let $|x|\geq |y|$. Note that on $V$, this is equivalent to $|z|\leq |x|$. So, we can represent $V$ as a graph with bounded 
gradient over $\C^2$ with coordinates $x,z$ under the map $y=G(x,z)=z^2/x$.
For the gradient, we get 
$\nabla G = (- z^2/x^2, 2z/x),$ 
which is in fact bounded as $|z|\leq |x|$. In the coordinates $x$, $z$ we have
$$\Psi(\omega) = - g \frac{dx \wedge dz}{\partial f/\partial y} = g \frac{dx\wedge dz}{x}$$
so that
$$|\Psi(\omega)|_V \lesssim \frac{1}{|x|}.$$
But this is in fact square-integrable in $\C^2$ with coordinates $x$, $z$ over the region $|z|\leq |x|$ where we have to integrate
(the pole of $1/x$ is only met in $0\in\C^2$).

The second chart is completely analogous by symmetry.
Let $|y|\geq |x|$. On $V$, this is equivalent to $|xy|\leq |y|^2$ or $|z|\leq |y|$. So, we can represent $V$ as a graph with bounded 
slope over $\C^2$ with coordinates $y,z$ under the map $x=G(y,z)=z^2/y$.
In the coordinates $y$, $z$ we have
$$\Psi(\omega) = g \frac{dy\wedge dz}{\partial f/\partial x} = - g \frac{dy\wedge dz}{y}$$
so that
$$|\Psi(\omega)|_V \sim \frac{1}{|y|}$$
is square-integrable over the region $|z|\leq |y|$ in $\C^2$.

That shows that $\Psi(\omega)$ is in fact square-integrable over $V$, thus a section in $\mathcal{K}_V$.

\end{example}


\begin{thebibliography}{99999}

























\bibitem[B]{B}{\sc B.\ Berndtsson}, $L^2$-extension of $\dq$-closed forms,
{\em Illinois J. Math.}, to appear. Available at {\sf arXiv:1104.4620}.







\bibitem[BM]{BiMi} {\sc E.\ Bierstone, P.\ Milman}, Canonical desingularization in characteristic zero by blowing-up
the maximum strata of a local invariant, {\em Inventiones Math.} {\bf 128} (1997), {\em no. 2}, 207--302.

















\bibitem[D1]{De3}{\sc J.-P.\ Demailly},
On the Ohsawa-Takegoshi-Manivel $L^2$ extension theorem, 
Complex analysis and geometry (Paris, 1997), Prog. Math. {\bf 188}, 47--82.


\bibitem[D2]{De2}{\sc J.-P.\ Demailly}, $L^2$ Hodge Theory and Vanishing Theorems,
in {\em Introduction to Hodge Theory}, SMF/AMS Texts and Monographs, vol. {\bf 8}, AMS (2002).


\bibitem[D3]{De4}{\sc J.-P.\ Demailly}, {\it Complex analytic and differential geometry}, Institut Fourier, Universit\'e de Grenoble I.






























\bibitem[GR]{GrRie} {\sc H.\ Grauert, O.\ Riemenschneider},
Verschwindungss\"atze f\"ur analytische Kohomologiegruppen auf komplexen R\"aumen,
{\em Inventiones Math.} {\bf 11} (1970), 263--292.




































\bibitem[H]{Hoe3}{\sc L.\ H\"ormander}, {\em An introduction to Complex Analysis in Several Variables},
Van Nostrand, Princeton, 1966.


\bibitem[K]{Kol}{\sc J.\ Koll\'ar}, Singularities of pairs. {\em Algebraic geometry -- Santa Cruz 1995}, 221--287,
Proc. Sympos. Pure Math. {\bf 62}, Part 1, Amer. Math. Soc., Providence, RI, 1997.




















\bibitem[M]{Ma}{\sc L.\ Manivel},
Un th\'eor\`eme de prolongement $L^2$ de sections holomorphes d'un fibr\'e vectoriel,
{\em Math. Z.} {\bf 212} (1993), 107--122.






























\bibitem[OT]{OT} {\sc T.\ Ohsawa, K.\ Takegoshi},
On the extension of $L^2$-holomorphic functions,
{\em Math. Z.} {\bf 195} (1987), no. 2, 197--204.

















\bibitem[PR]{PR}{\sc Th.\ Peternell, R.\ Remmert},
Differential calculus, holomorphic maps and linear structures on complex spaces.
{\em Several Complex Variables, VII}, 97--144, Encyclopaedia Math. Sci. {\bf 74}, Springer, Berlin, 1994.













\bibitem[R]{Rp8}{\sc J.\ Ruppenthal},
$L^2$-theory for the $\dq$-operator on compact complex spaces, 
{\em ESI-Preprint} {\bf 2202} (2010), Revised Preprint 2012. Available at {\sf arXiv:1004:0396},





































\bibitem[T]{T} {\sc K.\ Takegoshi}, Relative vanishing theorems in analytic spaces,
{\em Duke Math. J.} {\bf 51} (1985), no. 1, 273--279.










\end{thebibliography}
\end{document}